\documentclass{amsart}

\usepackage{color, mathtools, verbatim, esint}

\mathtoolsset{showonlyrefs}

\newcommand\blfootnote[1]{%
  \begingroup
  \renewcommand\thefootnote{}\footnote{#1}%
  \addtocounter{footnote}{-1}%
  \endgroup
}

\newcommand{\ddb}{\sqrt{-1}\partial\overline{\partial}}

\renewcommand{\[}{\begin{equation} \begin{aligned} }
      \renewcommand{\]}{\end{aligned} \end{equation}}
\renewcommand{\phi}{\varphi}

\newtheorem{thm}{Theorem}
\newtheorem{prop}[thm]{Proposition}
\newtheorem{lem}[thm]{Lemma}

\newtheorem{conj}[thm]{Conjecture}

\theoremstyle{definition}

\numberwithin{equation}{section}
\author{G\'abor Sz\'ekelyhidi}
\address{Department of Mathematics, University of Notre Dame, Notre Dame, IN 46556}
\email{gszekely@nd.edu}
\title{Uniqueness of some Calabi-Yau metrics on $\mathbf{C}^n$}
\date{}
\begin{document}
\begin{abstract}
We consider the Calabi-Yau metrics on $\mathbf{C}^n$ constructed
recently by Yang Li, Conlon-Rochon, and the author, that have tangent
cone $\mathbf{C}\times A_1$ at infinity for the $(n-1)$-dimensional
Stenzel cone $A_1$. We
show that up to scaling and isometry this Calabi-Yau
metric on $\mathbf{C}^n$ is unique. We also discuss possible
generalizations to other manifolds and tangent cones. 
\end{abstract}
\maketitle
\section{Introduction}
\blfootnote{The author is supported in part by NSF grants DMS-1350696
and DMS-1906216.}
On a compact K\"ahler manifold with vanishing first Chern class, Yau's solution of
the Calabi conjecture~\cite{Yau78} shows that any K\"ahler class
admits a unique Calabi-Yau metric. In the non-compact setting
there are many constructions
of complete Calabi-Yau manifolds with different asymptotic behaviors
by Cheng-Yau~\cite{CY80}, Tian-Yau~\cite{TY90,TY91} and others,
and even fixing the K\"ahler class these metrics are typically
not unique. To recover uniqueness, in general one needs to put 
conditions on the asymptotics of the metric. Our goal in this paper is
to prove such a uniqueness result for certain Calabi-Yau metrics on
$\mathbf{C}^n$. 

The Taub-NUT
metric on $\mathbf{C}^2$ is an example of a non-flat K\"ahler metric with
the same volume form as the Euclidean metric (see
LeBrun~\cite{LeBrun93}). This metric does not have maximal
volume growth and in fact the flat metric is the
unique Ricci flat metric on $\mathbf{C}^2$ with maximal volume growth (see
Tian~\cite{TianAspects}). It turns out that in higher dimensions this is no
longer the case, and for $n\geq 3$, $\mathbf{C}^n$ admits a complete
Calabi-Yau metric $\omega_0$
with tangent cone $\mathbf{C}\times A_1$ at infinity. Here $A_1$ is the
$(n-1)$-dimensional $A_1$ singularity $x_1^2 + \ldots + x_n^2 = 0$
equipped with the Stenzel cone metric
(see Li~\cite{Li17}, Conlon-Rochon~\cite{CR17} and the
author's work~\cite{Sz17}). These metrics all have the same volume
form as the Euclidean metric, and in fact there are infinitely many other
metrics with the same volume form, and different tangent cones at infinity.

It is therefore natural to try to classify Calabi-Yau metrics with a
prescribed tangent cone at infinity.
Classification results have previously
been obtained by Kronheimer~\cite{Kro89} in the case of surfaces, and
Conlon-Hein~\cite{CH14} in higher dimensions in the asymptotically
conical setting, i.e. when the metric converges at a polynomial rate to a Ricci
flat K\"ahler cone with smooth link. For instance in \cite{CH14} the
asymptotically conical Calabi-Yau manifolds with tangent cone $A_1$
are classified. 
Compared to these the main novelty in our work is that we are able
to deal with tangent cones that do not have isolated
singularities. 
Our main result is the following uniqueness statement for the metric
$\omega_0$ on $\mathbf{C}^n$.
\begin{thm}\label{thm:main}
  Suppose that $\omega$ is a complete Calabi-Yau
  metric on $\mathbf{C}^n$ with tangent cone $\mathbf{C}\times
  A_1$ at infinity. Then
  there is a biholomorphism $F:\mathbf{C}^n \to \mathbf{C}^n$
  and a constant $a > 0$ such that $\omega = a F^*\omega_0$.
\end{thm}

We emphasize that for a Calabi-Yau manifold with maximal volume
growth the tangent cone at infinity has a natural complex structure on
the regular set (which extends in general to the singular set by the
main results in \cite{DS15,LiuSz2}). When we say that the tangent cone is
$\mathbf{C}\times A_1$ we are requiring that the complex structures
agree as well as the metric structures,
since in principle there may be different complex structures on
a given metric cone. 

The proof of Theorem~\ref{thm:main} can likely be extended
to classify Calabi-Yau metrics on $\mathbf{C}^n$ with other tangent
cones, as well as $\partial\bar{\partial}$-exact Calabi-Yau metrics on
more general manifolds.  We will discuss this in Section~\ref{sec:further}.
The proof relies on two main ingredients. On
the one hand, given a $\partial\bar\partial$-exact
Calabi-Yau metric $(X,\omega)$ with tangent cone
$C(Y)$, the work of Donaldson-Sun~\cite{DS15} gives an algebraic
description of the ring of polynomial growth holomorphic functions on
$(X,\omega)$ in terms of the coordinate ring of $C(Y)$. When $C(Y) =
\mathbf{C}\times A_1$, and $X\cong \mathbf{C}^n$,
then we can use this description to obtain an
embedding $X\to\mathbf{C}^{n+1}$ as the hypersurface $z + x_1^2 +
\ldots + x_n^2 =0$, such that the functions $z, x_i$ have degrees $1,
\frac{n-1}{n-2}$ respectively. This is the basic input that allows us
to compare the unknown metric $(X,\omega)$ with the reference metric
$(\mathbf{C}^n, \omega_0)$, which is constructed by viewing
$\mathbf{C}^n\subset\mathbf{C}^{n+1}$ as the same hypersurface. We
discuss this in Section~\ref{sec:embed}. 

While we end up proceeding in a different way, heuristically the idea is that
using such an embedding we can hope to find a biholomorphism $F : \mathbf{C}^n
\to X$, such that $F^*\omega = \omega_0 + \ddb\phi$ satisfies the
Monge-Amp\`ere equation
\[ \label{eq:MA1} (\omega_0 + \ddb\phi)^n = \omega_0^n, \]
and in addition $\phi$ has subquadratic growth in the sense that
$r^{-2}\sup_{B(0,r)}|\phi| \to 0$ as $r\to \infty$. In practice we are
not able to do this, but if we could, we would then like
to show that $F^*\omega=\omega_0$, in analogy with the
uniqueness result of Conlon-Hein~\cite[Theorem 3.1]{CH13} in the
setting of asymptotically conical spaces. Instead we can only find a
sequence of such biholomorphisms $F$ on larger and larger balls.
The technical heart of the
proof is Proposition~\ref{prop:decayest}, which roughly
speaking says that if on some large $R$-ball we have a solution of
\eqref{eq:MA1} such that $R^{-2}\phi$ is small, then on a smaller
$\lambda R$-ball we can find an ``equivalent'' potential $\phi'$ such
that $(\lambda R)^{-2}\phi'$ is even smaller. The proof of this
result will take up most of Section~\ref{sec:decay}. Iterating this, and
letting $R\to \infty$, leads to Theorem~\ref{thm:main}.

Finally let us mention some related works for minimal
hypersurfaces. Regarding the uniqueness of minimal hypersurfaces with
prescribed tangent cone at infinity,  Simon-Solomon~\cite{SS86} and
Mazet~\cite{Maz17} showed that minimal hypersurfaces in
$\mathbf{C}^{n+1}$ that are asymptotic to certain Simons cones are
essentially unique. At the same time, the works by Simon~\cite{Sim94, Sim93},
and more recently Colombo-Edelen-Spolaor~\cite{CES17}, address the
behavior of minimal submanifolds that are near to a cone with
non-isolated singularities, which is also a key point in our case.
While the details are very different, there are certainly similarities
between our approach and theirs.

\subsection*{Acknowledgements} I would like to thank Nick Edelen and
Gang Liu for insightful discussions,
as well as Shih-Kai Chiu and Yang Li for helpful comments on an
earlier draft of the paper. 

\section{The reference metric}\label{sec:ref}
In this section we give some preliminary results about the Calabi-Yau
metrics on $\mathbf{C}^n$ constructed in \cite{Li17,
  CR17,Sz17}. We follow the approach from \cite{Sz17}. We suppose that
$f(x_1,\ldots, x_n)$ is a polynomial such that $V_0 = f^{-1}(0)\subset
\mathbf{C}^n$ has an isolated normal singularity at the origin. We
assume that $V_0$ admits a Calabi-Yau cone metric $\omega_{V_0} = \ddb
r^2$, whose homothetic
action is diagonal, with weights $(w_1,\ldots,w_n)$, and $f$ is
homogeneous of degree $d > 2$ under this action. The basic example we
are concerned with is $f = x_1^2 + \ldots + x_n^2$, in which case
$w_i=\frac{n-1}{n-2}$, and we
let $r^2 = |x|^{2\frac{n-2}{n-1}}$. In general it follows from
Conlon-Hein~\cite{CH13} (see also \cite[Section 2]{Sz17}) that the
smoothing $V\subset\mathbf{C}^n$ given by the equation $1+f(x)=0$
admits a Calabi-Yau metric $\omega_{V_1} = \ddb\phi(x)$, with tangent
cone $V_0$ at infinity.

We then consider the hypersurface $X\subset
\mathbf{C}\times\mathbf{C}^n$ given by $z+f(x)=0$, which is biholomorphic to
$\mathbf{C}^n$. The main result of \cite{Sz17} is that there exists a
Calabi-Yau metric $\omega_0$ on $\mathbf{C}^n$ with tangent cone
$X_0= \mathbf{C}\times V_0$ at infinity, which is uniformly equivalent to
the metric
\[ \omega = \ddb\Big( |z|^2 + \gamma_1(R\rho^{-\alpha})r^2 +
  \gamma_2(R\rho^{-\alpha})|z|^{2/d} \phi(z^{-1/d}\cdot x)\Big) \]
outside a compact set. Here $\gamma_i(s)$ are suitable cutoff
functions such that $\gamma_1+\gamma_2=1$, $\gamma_1$ is supported
where $s > 0$ while $\gamma_2$ is supported where $s < 2$; the
function $R$ is such that $\ddb R^2$ defines a cone metric on
$\mathbf{C}^n$ with the same homothetic action as $V_0$; the function
$\rho^2=|z|^2 + R^2$; $\alpha\in (1/d,1)$ and $z^{-1/d}\cdot x$ is defined using the
homothetic action, choosing a branch of $\log$. The form $\omega$
defines a metric when restricted to
$X$ outside of a compact set, and the Calabi-Yau metric
$\omega_0$ that is constructed is asymptotic to $\omega$ at infinity,
in the sense that $|\omega_0 - \omega|_\omega \to 0$. The volume form
of $\omega_0$ is $\sqrt{-1}^{n^2}\Omega\wedge\bar\Omega$ for
\[ \label{eq:Omega}\Omega = \frac{dz \wedge dx_2\wedge \ldots \wedge
    dx_n}{\partial_{x_1}f}. \]
For more details see \cite{Sz17}. 

From \cite[Proposition 9]{Sz17} we have the following. For large $D$,
we can consider a new embedding $X\to\mathbf{C}^{n+1}$ by the
functions $z' = D^{-1}z, x_i' =D^{-w_i}x_i$. The image has equation
\[ Dz' + D^d f(x') = 0, \]
i.e. $D^{1-d}z' + f(x') = 0$, recalling that $f$ has degree $d$ under
the homothetic action. We equip this hypersurface $X'$ with the scaled down
metric $D^{-2}\omega_0$. Here, and below, let us denote by
$\Psi(\epsilon)$ a function converging to zero as $\epsilon\to 0$. This function may
change from line to line. From \cite[Proposition 9]{Sz17} we see that
there is a constant $\theta < \Psi(D^{-1})$ satisfying the
following. We define the map $G: B_{X'}(0,1) \to X_0$ using the
nearest point projection on the set where $|x'| > \theta$, and
projection onto the $z$-axis where $|x'| \leq \theta$. Then $G$ is a
$\Psi(D^{-1})$-Gromov-Hausdorff approximation to $B_{X_0}(0,1)$. 
One useful consequence of this is that the distance from the origin in
$(X,\omega_0)$ is uniformly equivalent to the function $\rho$. We will
need the following. 

\begin{prop}\label{prop:vlingrowth}
  \begin{itemize}
    \item[(a)]
  The holomorphic functions $z, x_i$ on $(X,\omega_0)$ have polynomial
  growth with degrees $d(z)=1, d(x_i)=w_i$.
  \item[(b)] Consider the special case
  $f = x_1^2 + \ldots + x_n^2$. Then the vector fields $2z\partial_z +
  x_i\partial_{x_i} $ and $a_{jk} x_j \partial_{x_k}$
  for skew-symmetric $(a_{jk})$ all have at most linear growth. 
  \end{itemize}
\end{prop}
\begin{proof}
  The statement in (a) is immediate from the fact that the distance
  function is uniformly equivalent to $\rho$.

  For part (b), we can work using the description of the metric
  $\omega$ in regions I--V in the proof of \cite[Proposition
  5]{Sz17} (note that $\omega$ is uniformly equivalent to
  $\omega_0$). In each region we choose new coordinates in which we
  have a good model for the form $\omega$ and so we can bound our
  vector fields. Let us consider regions I, III and V, the others being
  very similar. 

  \bigskip
  {\bf Region I:} Here $R > \kappa\rho$ for some fixed small $\kappa > 0$,
  and we assume $\rho\in (D/2, 2D)$ for $D$, which will then be
  uniformly equivalent to the distance from the origin. We change
  coordinates to $\tilde{z} = D^{-1}z$ and $\tilde{x} = D^{-1}\cdot
  x$, and we let $\tilde{r} = D^{-1}r$. In these coordinates $X$ has
  equation
  \[ \label{eq:D1} D^{1-d}\tilde{z} + f(\tilde{x}) = 0, \]
  and in the proof of \cite[Proposition
  5]{Sz17} the scaled down metric
  $D^{-2}\omega$ on this hypersurface is compared to the product metric
  \[ \label{eq:D2w} \ddb(|\tilde{z}|^2 + \tilde{r}^2), \]
  on the hypersurface with equation $f(\tilde{x})=0$ (i.e. the product
  $X_0$). Since
  $|\tilde{z}| < 2$ and $\tilde{r}\in (\kappa/2,
  4\kappa)$, as well as $d > 1$, as $D\to\infty$ then in these
  coordinates the hypersurface \eqref{eq:D1} converges smoothly to
  $X_0$. Because of this we can compute the norms of our vector fields
  with respect to the metric \eqref{eq:D2w}.  For this we have
  \[ z\partial_z = \tilde{z}\partial_{\tilde{z}}, \quad
    x_j\partial_{x_k} = \tilde{x}_j \partial_{\tilde{x}_k}. \]
  The norms of these vector fields are uniformly bounded for the
  metric in \eqref{eq:D2w}, which is uniformly equivalent to
  $D^{-2}\omega$ (under identifying the two hypersurfaces), and so
  \[ |z\partial_z|_\omega, |x_j\partial_{x_k}|_\omega < CD, \]
  for a constant $C$. 

  \bigskip
  {\bf Region III:} Here $R\in (K/2, 2K)$, and $K\in(\rho^\alpha,
  2\rho^\alpha)$. We suppose $\rho\in (D/2, 2D)$, so $|z|$ is
  comparable to $D$. We choose a fixed $z_0$ such that $|z-z_0|< K$,
  and we change variables as follows:
  \[ \tilde{z} = K^{-1}(z-z_0), \quad \tilde{x} = K^{-1}\cdot x, \quad
    \tilde{r} = K^{-1}r. \]
  In these coordinates $X$ is given by the equation
  \[ K^{-d}(K\tilde{z} + z_0) + f(\tilde{x}) = 0, \]
  and we compare again to the product metric 
  \[ \ddb(|\tilde{z}|^2 + \tilde{r}^2), \]
  on the hypersurface $f(\tilde{x})=0$. We have $|\tilde{z}| < 1$,
  $\tilde{R}\in (1/2, 2)$, and $K^{-d}z_0\to 0$ as $K\to\infty$, since
  $K^d \gg D$. This means that as $K,D\to\infty$, we can measure the
  norms of our vector fields on $X_0$ with the product metric. 
  We  have
  \[ z\partial_z = (K\tilde{z} + z_0)K^{-1} \partial_{\tilde{z}},
    \quad x_j\partial_{x_k} = \tilde{x}_j \partial_{\tilde{x}_k}. \]
  It follows that
  \[ |z\partial_z|_{K^{-2}\omega} < CK^{-1}D, \quad
    |x_j\partial_{x_k}|_{K^{-2}\omega} < C. \]
  Since $D$ is comparable to the distance from the origin, and $K\ll
  D$, this implies the estimate we want. 

  \bigskip
  {\bf Region V:} Here $R < 2\kappa^{-1}\rho^{1/d}$,
  $\rho\in(D/2,2D)$, so $|z|$ is comparable to $D$. We choose a fixed
  point $z_0$ with $|z-z_0| < D^{1/d}$, so we also have $|z_0|\sim
  D$. We scale our metric down by a factor of $|z_0|^{1/d}$, and
  change coordinates by
  \[ \tilde{z}= z_0^{-1/d}(z-z_0), \quad \tilde{x} = z_0^{-1/d}\cdot
    x, \quad \tilde{r} = |z_0|^{-1/d} r. \]
  We have $|\tilde{z}|, \tilde{r} < C$ for a uniform $C$. The equation
  of $X$ is
  \[ z_0^{1/d-1}\tilde{z} + 1 + f(\tilde{x}) = 0, \]
  and as $D\to \infty$, this converges to hypersurface with equation
  \[ 1 + f(\tilde{x}) = 0, \]
  and the metric $|z_0|^{-2/d}\omega$ converges to
  \[ \label{eq:m1} \ddb\Big(
  |\tilde{z}|^2 + \phi(\tilde{x})\Big). \] We have
  \[ z\partial_z = z_0^{-1/d}(z_0^{1/d}\tilde{z} + z_0)
    \partial_{\tilde{z}}, \quad x_j\partial_{x_k} =
    \tilde{x}_j\partial_{\tilde{x}_k}. \]
  The norms of these vector fields are uniformly bounded with respect
  to \eqref{eq:m1}, and so scaling back up, we have
  \[ |z\partial_z|_\omega < |z_0|^{1/d}C|z_0|^{-1/d}|z_0| < CD, \quad
    |x_j\partial_{x_k}|_\omega < C|z_0|^{1/d} < CD^{1/d},  \]
  which gives the required bound. 
\end{proof}

\subsection{Subquadratic harmonic functions on $\mathbf{C}\times A_1$}
Let us consider the tangent cone $C(Y) = \mathbf{C}\times
A_1$ embedded in $\mathbf{C}\times \mathbf{C}^n$ as the
hypersurface $x_1^2 +\ldots +x_n^2=0$, and equipped with the Stenzel
cone metric $\ddb(|z|^2 + |x|^{2\frac{n-2}{n-1}})$. We need to understand the harmonic
functions on $C(Y)$ with at most quadratic growth. In
Hein-Sun~\cite{HS16} a general result is given on Calabi-Yau cones
with isolated singularities, saying that the strictly subquadratic harmonic
functions are all pluriharmonic (this was first used crucially in
Conlon-Hein~\cite{CH13}), while the space of exactly quadratic 
growth harmonic functions decomposes as the sum of pluriharmonic
functions and harmonic functions that arise from isometries of the
link. See also Chiu~\cite{Chiu19} for results in the case of more
singular cones.  We have the following.

\begin{lem}\label{lem:CYharmonic}
  The space $\mathcal{H}_{\leq 2}$ of
  real harmonic functions on $C(Y)$ with at most quadratic growth
  are given by linear combinations of the following:
  \begin{enumerate}
  \item the real and imaginary parts of $1,z,z^2,x_i$,
  \item the function $(n-1)|z|^2-|x|^{2\frac{n-2}{n-1}}$,
  \item the functions $|x|^{-\frac{2}{n-1}}a_{jk}x_j\bar{x}_k$, where $(a_{jk})\in
    \sqrt{-1}\mathfrak{o}(n,\mathbf{R})$ is a purely imaginary complex
    orthogonal matrix.
  \end{enumerate}
\end{lem}
\begin{proof}
  A general approach to this result is to extend
  Hein-Sun~\cite[Theorem 2.14]{HS16} to singular tangent cones. This
  can be done along the lines of the work in Chiu~\cite{Chiu19}, using
  cutoff functions to justify the required integration by parts near
  the singular set.

  Alternatively we can follow the approach from \cite[Corollary
  12]{Sz17} using the Fourier transform in the $\mathbf{C}$-direction
  to analyze harmonic functions on the product $\mathbf{C}\times
  A_1$. The conclusion from this approach is that any harmonic
  function $f$ of at most quadratic growth can be written as
  \[ f = f_0 + zf_1 + \bar{z}f_{\bar 1} + z^2 f_{2} + \bar{z}^2
    f_{\bar 2} + |z|^2 f_{1\bar 1}, \]
  for functions $f_0, f_1, f_{\bar 1}, f_2, f_{\bar 2}, f_{1\bar 1}$
  on the cone $A_1$. We have
  \[ \Delta f = \Delta' f_0 + z\Delta' f_1 + \bar{z} \Delta' f_{\bar
      1} + z^2 \Delta' f_{2} + \bar{z}^2 \Delta' f_{\bar 2} + |z|^2
    \Delta' f_{1\bar 1} + f_{1\bar 1}, \]
  where $\Delta'$ is the Laplacian on $A_1$. It follows from $\Delta f
  = 0$ that
  \[ \Delta' f_1 = \Delta' f_{\bar 1} = \Delta' f_2 &= \Delta' f_{\bar
      2} = \Delta' f_{1\bar 1} = 0, \\
    f_{1\bar 1} + \Delta' f_0  &= 0. \]
  In addition since $f$ has at most quadratic growth, $f_1, f_{\bar
    1}, f_2, f_{\bar 2}$ are all subquadratic harmonic functions, so
  by \cite[Theorem 2.14]{HS16} they are pluriharmonic. Since the
  non-constant holomorphic functions on $A_1$ have faster than linear
  growth, these functions must all be constant. The function $f_{1\bar
    1}$ is harmonic, and $|z|^2f_{1\bar 1}$ has at most quadratic
  growth. It follows that $f_{1\bar 1} = c$ is constant. Then
  \[ c + \Delta' f_0 = 0, \]
  so $f_0' = (n-1)f_0 - c|x|^{2\frac{n-2}{n-1}}$ is harmonic, and has at most
  quadratic growth. Using \cite[Theorem 2.14]{HS16} again, we have
  that $f_0'$ is a linear combination of real and imaginary parts of
  $1,x_i$, and functions $u$ such that $V = \nabla u$ is a real holomorphic
  vector field on $A_1$ commuting with $r\partial_r$ such that
  $JV(r)=0$. We then have
  $V = \mathrm{Re}(a_{jk}x_j\partial_{x_k})$ for $a_{jk}$ a purely
  imaginary skew symmetric matrix. Using the identity $\ddb u =
  L_{\nabla u} \ddb r^2 = \ddb V(r^2)$,  up to adding a pluriharmonic
  function to $u$, we have
  \[ u = V(r^2) = V(|x|^{2\frac{n-2}{n-1}}) = \frac{n-2}{n-1}
    |x|^{-\frac{2}{n-1}} a_{jk}x_j\bar{x}_k. \]
  The result follows from this. 
\end{proof}
The functions in (1) are all the pluriharmonic functions of at most
quadratic growth, while (2) and (3) correspond
to automorphisms of $C(Y)$ commuting with the homothetic scaling,
which has weights $(1,\frac{n-1}{n-2},\ldots,\frac{n-1}{n-2})$ on
$(z,x_1,\ldots,x_n)$. As in the proof, the functions $|x|^{-\frac{2}{n-1}}a_{jk}x_j\bar{x}_k$
correspond to the vector fields $V_{\mathbf a} =
\frac{n-1}{n-2}\mathrm{Re}(a_{jk}x_j\partial_{x_k})$, in the sense
that
\[ V_{\mathbf{a}}(|z|^2 + |x|^{2\frac{n-2}{n-1}}) =
  |x|^{-\frac{2}{n-1}}a_{jk}x_j\bar{x}_k. \]
These vector fields preserve the hypersurfaces $cz + x_1^2 + \ldots
+x_n^2 = 0$ for all $c$, and the volume form $\Omega$. 

Similarly, the function
$(n-1)|z|^2-|x|^{2\frac{n-2}{n-1}}$ corresponds to the real holomorphic vector field
$W=\mathrm{Re}((n-1)z\partial_z - \frac{n-1}{n-2}x_i\partial_{x_i})$,
i.e. 
\[ W(|z|^2 + |x|^{2\frac{n-2}{n-1}}) =
  (n-1)|z|^2-|x|^{2\frac{n-2}{n-1}}. \]
This vector field $W$
preserves $C(Y)\subset\mathbf{C}^{n+1}$, however it does not preserve
the hypersurfaces $cz + x_1^2 + \ldots + x_n^2=0$. Instead we let
\[ V = \mathrm{Re}( z\partial_z + \frac{1}{2} x_i\partial_{x_i}), \]
which does preserve all of these hypersurfaces. The vector field $V$
satisfies
\[ L_V \Omega = \frac{n}{2}\Omega, \]
and
\[ V(|z|^2 + |x|^{2\frac{n-2}{n-1}}) - \frac{1}{2}(|z|^2 +
  |x|^{2\frac{n-2}{n-1}}) = \frac{1}{2}|z|^2 - \frac{1}{2(n-1)}
  |x|^{2\frac{n-2}{n-1}}, \]
which is a scalar multiple of the function in (2). We conclude the following.
\begin{lem}\label{lem:vbeta}
  Suppose that $h$ is a harmonic function on $C(Y)$ with at most
  quadratic growth, and write $h=h_{ph} + h_{aut}$, where $h_{ph}$ is
  in the span of the type (1) functions in Lemma~\ref{lem:CYharmonic},
  and is pluriharmonic, 
  while $h_{aut}$ is in the span of the type (2) and (3)
  functions.

  We can find a real holomorphic vector field $V$ preserving the
  hypersurfaces $cz + x_1^2 + \ldots + x_n^2=0$, and a constant
  $\beta$ such that $L_V\Omega = n\beta\Omega$, and
  \[ V(|z|^2 + |x|^{2\frac{n-2}{n-1}}) - \beta(|z|^2 +
    |x|^{2\frac{n-2}{n-1}}) = h_{aut}. \]
  In addition we have $|\beta| \leq C\Vert h\Vert$ and
  $V = a_0 z\partial_z + a_{jk}x_j\partial_{x_k}$ with $|a_0|,
  |a_{jk}| \leq C\Vert h\Vert$, for a constant $C$,
  where $\Vert h\Vert$ denotes the $L^2$ norm on $B(0,1)\subset C(Y)$.  
\end{lem}

\section{Special embeddings}\label{sec:embed}

In this section $(X,\eta)$ is a complete Calabi-Yau manifold such
that $X$ is biholomorphic to $\mathbf{C}^n$, and $X$ has tangent cone
$C(Y) = \mathbf{C}\times A_1$ at infinity. Let us fix a
basepoint $p\in X$, and denote by $B_i$ the ball $B(p, 2^i)$, with the
metric $2^{-2i}\eta$, so that $B_i$ is a unit ball in the scaled
down metric. By assumption, the sequence $B_i$ converges to the unit
ball $B(0,1) \subset C(Y)$ in the Gromov-Hausdorff sense. We will
view $C(Y)\subset \mathbf{C}^{n+1}$ as defined by the equation
$x_1^2+\ldots+x_n^2=0$, in terms of the coordinates $z,x_i$ on
$\mathbf{C}^{n+1}$ . The cone $C(Y)$ is equipped with the Ricci flat
Stenzel metric given by
\[ \ddb( |z|^2 + |x|^{2\frac{n-2}{n-1}} ), \]
which has volume form $\sqrt{-1}^{n^2}\Omega\wedge\bar\Omega$, in
terms of the $\Omega$ from \eqref{eq:Omega}. 
  The main result of this section is the following.
\begin{prop}\label{prop:goodembed}
  There is a sequence of holomorphic embeddings
  \[ F_i : X \to \mathbf{C}^{n+1}, \]
  with the following properties:
  \begin{enumerate}
    \item the image $F_i(X)$ is given by the equation
      \[ a_iz + x_1^2 + \ldots + x_n^2 = 0, \]
      for some $a_i > 0$,
    \item the volume form $\eta^n$ satisfies
      \[ 2^{-2ni}\eta^n = F_i^*(\sqrt{-1}^{n^2}\Omega\wedge\bar\Omega), \]
    \item  $a_i / a_{i+1} \to 2^{n/(n-2)}$ as $i\to\infty$,
    \item  on the ball $B_i$ the map $F_i$ gives a
      $\Psi(i^{-1})$-Gromov-Hausdorff approximation to the embedding
      $B(0,1)\to \mathbf{C}^{n+1}$. More precisely, we have a
      $\Psi(i^{-1})$-Gromov-Hausdorff approximation $g: B_i \to
      B(0,1)$, such that $|F_i- g| < \Psi(i^{-1})$ on $B_i$.
      Recall that here $B(0,1)\subset
      \mathbf{C}^{n+1}$ is the unit ball of $C(Y)$ under our
      embedding, and $\Psi(i^{-1})$ denotes a function converging to
      zero as $i\to\infty$. 
   \end{enumerate}
 \end{prop}

The main input for this result is the work of Donaldson-Sun~\cite{DS15}
on the algebro-geometric study of tangent cones, and we first review
the results that we use.

\subsection{Donaldson-Sun theory} In \cite{DS12, DS15}, Donaldson-Sun
consider non-collapsed Gromov-Hausdorff limits of compact polarized
K\"ahler manifolds with bounded Ricci curvature. We observe that for
many of the arguments compactness is not required (see also
Liu~\cite{Liu17, Liu16} for 
related work in the non-compact setting). More precisely, 
suppose that $(M_i, L_i, \omega_i, p_i)$ is a sequence of complete
pointed $n$-dimensional
K\"ahler manifolds with line bundles $L_i\to M_i$ equipped with Hermitian
metrics with curvature $-\sqrt{-1}\omega_i$. In addition suppose that
we have the Einstein condition $\mathrm{Ric}(\omega_i) =
\lambda_i\omega_i$ with $|\lambda_i|\leq 1$, and the non-collapsing
condition $\mathrm{Vol}(B(p_i, 1)) > \kappa > 0$ for all $i$, for a fixed
$\kappa > 0$. If in addition we were to assume that the manifolds were
compact, then the sequence would be in the class
$\mathcal{K}(n,\kappa)$ considered in \cite{DS15}.

Let us suppose that $(Z, p)$ is the pointed Gromov-Hausdorff limit of
the sequence $(M_i, \omega_i, p_i)$. Then \cite[Theorems 1.1, 1.2,
1.3]{DS15} hold, i.e. $Z$ has the structure of a normal complex
analytic space, and tangent cones to $Z$ have the structure of affine
varieties and are unique. To see this note that the basic construction in
\cite{DS12} that is used in the arguments is to ``graft'' a holomorphic function from
a tangent cone to $Z$, using cutoff functions,
onto $M_i$ for sufficiently large $i$, and then use the H\"ormander
$L^2$-estimate the perturb the resulting approximately holomorphic
section of (a power of) $L_i$ to a holomorphic section $s$. The grafting
is a local construction, and the H\"ormander estimate holds on 
complete K\"ahler manifolds (see e.g. \cite[Theorem 4.5]{Demaillybook}). Finally
one uses Moser iteration and
Bochner-Weitzenbock type formulas to bound the $L^\infty$
norms of $s$ and $\nabla s$ in terms of the $L^2$-norm of $s$ (see
\cite[Proposition 2.1]{DS12}). Since under the non-collapsing
condition and Ricci curvature lower bound we can control the Sobolev
constant on geodesic balls (see e.g. Anderson~\cite[Theorem
4.1]{And92}), the same estimates hold in our setting. 

We need to use the results in \cite[Section 3.4]{DS15}, however the
assumptions there are that the limit space $(Z,p)$ is a scaled limit
of a sequence in $\mathcal{K}(n, \kappa)$, with scaling factors
tending to infinity. We claim, however, that the same results hold for
a complete Calabi-Yau manifold $(M,\omega)$, where $\omega=\ddb\psi$
for a global K\"ahler potential $\psi$, which has maximal volume growth:
$\mathrm{Vol}(B(p,r))> \kappa r^{2n}$ for all $r > 0$, for a
basepoint $p\in M$. The basic reason is that in this case the tangent
cone at infinity is still the Gromov-Hausdorff limit of a sequence of
polarized K\"ahler manifolds as above: 
for any sequence $\lambda_i\to 0$, we can 
consider the sequence $(M_i, \omega_i, L_i, p_i)$, where $M_i=M,
\omega_i = \lambda_i^2\omega, p_i=p$, and $L_i$ is the trivial bundle
equipped with the metric $e^{-\lambda_i^2\psi}$. Up to choosing a
subsequence, this sequence converges in the Gromov-Hausdorff sense
to a tangent cone at infinity $C(Y)$ of $(M,\omega)$. Using this, the
arguments in \cite[Section 2.2]{DS15} can be applied to the limit
space $C(Y)$ (instead of $(Z,p)$ in the statements of the propositions
there)
without any changes. In particular \cite[Proposition 2.9]{DS15} holds,
showing that holomorphic functions on a ball in $C(Y)$ can be
approximated by holomorphic functions on suitable balls in $M_i$. This
is a crucial ingredient in \cite[Proposition 3.26]{DS15}, which leads
to the algebro-geometric description of the tangent cone $C(Y)$ in
terms of the ring of polynomial growth holomorphic functions on
$(M,\omega)$. 

Let us  briefly recall the results that we need from \cite[Section
3.4]{DS15}, where for us $M$ plays the role of $Z$ there.
If $R(M)$ denotes the ring of polynomial
growth holomorphic functions on $M$, then $R(M)$ has a filtration
\[ \mathbf{C} = I_0 \subset I_1 \subset \ldots \subset R(M). \]
Here $I_k$ is the space of polynomial growth
holomorphic functions on $M$ with degree
at most $d_k$, where $0=d_0 < d_1 < \ldots$ are the possible growth
rates. For a holomorphic function $f$ on $X$, the growth rate
$d(f)$ is defined by
\[ d(f) = \lim_{r\to\infty} (\log r)^{-1} \sup_{B(p,r)} \log |f|, \]
and $f$ has polynomial growth if $d(f)  < \infty$. 
These growth rates are the same as the possible
growth rates on the tangent cone $C(Y)$, and the dimensions $\dim I_k$
are equal to the corresponding dimensions on $C(Y)$. Let us
write $R_{d_k}$ for the functions of degree $d_k$ on $C(Y)$, and
$\mu_k = \dim R_{d_k}$. 

By \cite[Proposition 3.26]{DS15} we can find decompositions
$I_k = I_{k-1} \oplus J_k$, where $\dim J_k = \mu_k$, and $J_k$ admits
an adapted sequence of bases. This means that, for fixed $k$, we have
a sequence of bases $\{G^i_1,\ldots, G^i_{\mu_k}\}$ for $J_k$, satisfying
\begin{enumerate}
\item $\Vert G^i_a\Vert_{B_i}=1$ for all $a$, and for $a\ne b$ we have
  $\lim_{i\to\infty} \int_{B_i} 
  G_a^i\overline{G_b^i} = 0$. Here
  $\Vert\cdot\Vert_{B_i}$ denotes the $L^2$-norm on $B_i$, and as
  above, $B_i$ is the ball $B(p, 2^i)$ scaled down to unit size.
\item $G_a^{i+1} = \mu_{ia}G_a^i + p^i_a$ for scalars $\mu_{ia}$, and
  $p^i_a\in \mathbf{C}\langle G^i_1,\ldots G^i_{a-1}\rangle$, with
  $\Vert p^i_a\Vert_{B_i}\to 0$ as $i\to\infty$. 
\item $\mu_{ia} \to 2^{d_k}$ as $i\to\infty$. 
\end{enumerate}
Suppose now that the coordinate ring $R(C(Y))$ is generated by
$\bigoplus_{k\leq k_0} R_{d_k}$. It follows then that $R(M)$ is generated
by $\bigoplus_{k\leq k_0} J_k$, and the adapted bases of $J_k$ for
$k\leq k_0$ define embeddings $F_i : M \to \mathbf{C}^N$. Furthermore,
under the Gromov-Hausdorff convergence $B_i\to B(0,1)\subset C(Y)$,
the maps $F_i$ converge to an embedding $B(0,1)\to \mathbf{C}^N$ by an
$L^2$-orthonormal basis of $R_{d_0} \oplus \ldots \oplus
R_{d_{k_0}}$. 

\subsection{Proof of Proposition~\ref{prop:goodembed}}
We now specialize to the setting of Proposition~\ref{prop:goodembed}. 
It will be helpful to write down the homogeneous holomorphic functions
of low degree on $C(Y)=\mathbf{C}\times A_1$.
Note that they are all spanned by
polynomials in $z,x_i$, and $d(z)=1, d(x_i)=\frac{n-1}{n-2}$. We treat
three cases separately:
\begin{itemize}
\item $n=3$. In this case we have
\[ R_0 &= \langle 1\rangle, \\
  R_1 &= \langle z\rangle,\\
  R_2 &= \langle z^2, x_i\rangle, \\
  R_3 &= \langle z^3, zx_i\rangle, \\
  R_4 &= \langle z^4, z^2x_i, x_ix_j \rangle, \]
where in $R_4$ one term is redundant because of the equation $x_1^2 +
\ldots + x_n^2 =0$. 
\item $n=4$. Here we have
  \[ R_0 &= \langle 1\rangle, \\
    R_1 &= \langle z\rangle,\\
    R_{3/2} &= \langle x_i\rangle, \\
    R_{2} &= \langle z^2\rangle, \\
    R_{5/2} &= \langle zx_i\rangle, \\
    R_3 &= \langle z^3, x_ix_j \rangle, \]
  where again one term in $R_3$ is redundant.
\item $n>4$.
    \[ R_0 &= \langle 1\rangle, \\
      R_1 &= \langle z\rangle,\\
      R_{\frac{n-1}{n-2}} &= \langle x_i\rangle,\\
      R_2 &= \langle z^2\rangle,\\
      R_{\frac{2n-3}{n-2}} &= \langle zx_i\rangle,\\
      R_{\frac{2n-2}{n-2}} &= \langle x_ix_j\rangle, \]
      where again one term in $R_{\frac{2n-2}{n-2}}$ is redundant. 
\end{itemize}

For simpicity we focus on the case $n=3$. The discussion in the other cases
is completely analogous. The ring $R(C(Y))$ is generated by $R_1\oplus
R_2$, and so 
by the results of Donaldson-Sun~\cite{DS15} discussed above, $R(X)$ is
generated by $I_2 = J_0\oplus J_1 \oplus J_2$. This space 
of holomorphic functions
on $X$ with at most quadratic growth has $\dim I_2 = 6$, and 
admits a sequence of adapted bases $\{G^i_1,\ldots, G^i_6\}$. 
The growth
rates of the functions are $d(G^i_1)=0, d(G^i_2)=1$ and $d(G^i_j)=2$
for $j=3,4,5,6$, and for each $i$ we obtain an embedding
\[ F_i : X \to \mathbf{C}^6, \]
with components $G^i_j$. On the balls $B_i$ these maps converge in
the Gromov-Hausdorff sense to an embedding $F_\infty
:B(0,1)\to\mathbf{C}^6$. The map $F_\infty$ is given by
functions on $C(Y)$ with the degrees specified above, that are
orthonormal on $B(0,1)$. Up to a unitary transformation commuting
with the homothetic scaling (which has degrees $(0,1,2,2,2,2)$) we can
assume that $F_\infty = (1,z,z^2,x_1,x_2,x_3)$. We can modify our
sequence of adapted bases by the same unitary transformation, so that
we still have $F_i\to F_\infty$ as $i\to\infty$. 

Since $I_0$ consists of just the constants, the first component of
$F_i$ is constant. In
addition, we have $\Vert (G^i_2)^2 - G^i_3\Vert_{B_i} \to 0$ by the convergence of
$F_i$ to $F_\infty$, while also $(G^i_2)^2 - G^i_3 \in I_2$. It
follows that $G^i_3 = (G^i_2)^2 + \sum_a q_{ia}G^i_a$, where $|q_{iq}|
< \Psi(i^{-1})$. Therefore dropping the first and third components of $F_i$, we
still obtain embeddings $F_i' : X \to \mathbf{C}^4$. Let us now fix
$i$, and abusing notation, let us denote by $z, x_1, x_2, x_3$ the
pullbacks under $F_i'$ to $X$ of the coordinate functions on
$\mathbf{C}^4$. By construction we have $d(z)=1, d(x_i)=2$. The 20 
functions
\[ 1, z, z^2, x_i, z^3, zx_i, z^4, z^2x_j, x_jx_k \]
are all in the space $I_4$ of at most quartic growth functions on $X$,
but by the earlier discussion $\dim I_4 = 19$. Therefore we have one
linear dependency between them, which determines the equation $f_i$ in
$\mathbf{C}^4$ defining $F_i'(X)$. Because of the Gromov-Hausdorff
convergence of the $F_i'$ to the embedding of $B(0,1)$ satisfying
$x_1^2+x_2^2+x_3^2=0$, the equation $f_i(z,x_1,x_2,x_3)=0$
has to be a perturbation of this equation. We can apply a
linear transformation in $x_1,x_2,x_3$ that is $\Psi(i^{-1})$-close to the identity,
to transform the quadratic expression in the $x_j$ that appears in
$f_i$ to the quadric $x_1^2+x_2^2+x_3^2$. Next we can complete the
square in the $x_j$, applying changes of coordinates of the form
$x_j\mapsto x_j + a_jz^2 + b_jz + c_j$ with small $a_j, b_j, c_j$, to eliminate the
terms of the form $x_j, zx_j, z^2x_j$ in $f_i$. We have now reduced our
equation to one of the form
\[ \label{eq:reducedeq} \tilde{f}_i(z) + x_1^2+ x_2^2+x_3^2 = 0, \]
where $\tilde{f}_i(z)$ is a quartic polynomial in $z$ with
coefficients of order $\Psi(i^{-1})$. Since $X$ is
biholomorphic to $\mathbf{C}^3$, $\tilde{f}_i$ must actually be
linear, so $f_i(z) = d_iz + e_i$ for $|d_i|, |e_i| < \Psi(i^{-1})$
with $d_i\ne 0$. We can assume that $d_i > 0$ by multiplying $z$ by a
unit complex number. 

So far we have
only performed coordinate changes $\Psi(i^{-1})$-close to the
identity,  and so the new maps $F_i'':X\to\mathbf{C}^4$
that we obtain still converge on $B_i$, in the Gromov-Hausdorff sense,
to our standard embedding $B(0,1) \to\mathbf{C}^4$. The image
$F_i''(X)$ satisfies the equation
\[ e_i + d_i z + x_1^2 + x_2^2 + x_3^2=0. \]
We want to perform a further coordinate change $z \mapsto z + d_i^{-1}e_i$,
so that the equation reduces to $d_i z + x_1^2+x_2^2+x_3^2=0$,
however this may not be a small coordinate change, since we do not
yet have information about the relative sizes of $d_i,e_i$. For this
we need the following.

\begin{lem} \label{lem:o}
  Suppose that we have polynomial growth holomorphic functions $w,
  y_1,y_2,y_3$ on $X$ such that $d(w)=1$ and $d(y_j)=2$, and 
  \[ cw + y_1^2+y_2^2+y_3^2=0, \]
  for some $c\ne 0$. Then there exists a point $o\in X$ such that
  $w(o) = y_j(o) = 0$, and moreover this point is independent of the
  choice of functions $w,y_j$ as above.
\end{lem}
\begin{proof}
  In the above discussion we have already constructed functions
  $z,x_j$ satisfying
  \[ c'z + x_1^2 + x_2^2 + x_3^2 = 0. \]
  Let us denote by $o\in X$ the point where $z(o)=x_i(o)=0$. 
  In terms of these functions, because of the degree restrictions, we must have
  \[ w &= az+b, \\
    \mathbf{y} & =A\mathbf{x} + \mathbf{c}z^2 +\mathbf{d}z  + \mathbf{e}. \]
  By scaling we can assume that $c=c'=1$ to simplify notation. We have
  \[ \label{eq:wy}
    0 = w + \mathbf{y}^T \mathbf{y} = az + b + (A\mathbf{x} + \mathbf{c}z^2 +\mathbf{d}z  + \mathbf{e})^T(A\mathbf{x} + \mathbf{c}z^2 +\mathbf{d}z  + \mathbf{e}). \]
  By our dimension counts earlier, up to a scalar multiple,
  there is only one linear equation
  satisfied by the functions
  \[ 1,z,z^2,x_j, z^3, zx_j, z^4, z^2x_j, x_jx_k, \]
  namely $z + \mathbf{x}^T\mathbf{x}=0$. Therefore if we also have $w
  + \mathbf{y}^T\mathbf{y}=0$, then from \eqref{eq:wy} we first find
  that $A^TA$ is a multiple of the identity, and in particular $A$ is
  invertible. It then follows that we have
  $\mathbf{c}=\mathbf{d}=\mathbf{e}=0$, and in turn $b=0$. Therefore
  $w(o)=y_j(o)=0$, and conversely $o$ is the unique point where $w,
  y_j$ all vanish. 
\end{proof}

Let us return to the proof of Proposition~\ref{prop:goodembed}, and the
map $F_i'':X\to\mathbf{C}^4$ whose image satisfies the equation
\[ d_i(z + d_i^{-1}e_i) + x_1^2+ x_2^2+x_3^2=0. \]
By the Lemma above we have $z(o) = -d_i^{-1}e_i$, however since $F_i$
converges to the embedding $B(0,1)\to\mathbf{C}^4$, we also have
$F_i(o)\to 0$. For this note that under the rescaled metric in the
ball $B_i$ (centered at $p$) the point $o$ is contained in a ball of
radius $C2^{-i}$ around $p$
where $C = \mathrm{dist}_\eta(p,o)$. In particular we have
$d_i^{-1}e_i\to 0$, and so we can perform a final small coordinate
change to reduce to the case $e_i=0$.

Next we consider the volume forms. The pullback $(F_i'')^*(\Omega)$
defines a polynomial growth, nowhere vanishing holomorphic volume form
on $X$. By the Calabi-Yau condition on $\eta$
\[ F_i''^*(\sqrt{-1}\Omega\wedge\bar\Omega) = |g_i|^2 \eta^3, \]
where $g_i$ is a nowhere vanishing polynomial growth holomorphic
function on $X$. This means $g_i$ can be written as a polynomial in
$x_1,x_2,x_3$, and therefore it is constant since it has no
zeros. By the volume convergence under Gromov-Hausdorff
convergence~\cite{Col97}, remembering that we are using the scaled
down metric $2^{-2i}\eta$ in the convergence, we find that $|g_i|^2 2^{6i}\to 1$ as
$i\to\infty$. Scaling $z$ by a factor $\Psi(i^{-1})$-close to 1, we can arrange that
$|g_i|^2=2^{-6i}$. 

Finally we address the claim that $a_i/a_{i+1}\to 8$.
Let us denote the components of our maps $F_i''$
by $z_i, x_{i,1}, x_{i,2}, x_{i,3}$, so that
\[ a_iz_i + x_{i,1}^2 + x_{i,2}^2 + x_{i,3}^2 = 0, \]
with $a_i > 0$. 
From the proof of Lemma~\ref{lem:o} we know that
\[ z_{i+1} &= c_i z_i, \\
  \mathbf{x}_{i+1} &= \sqrt{\frac{a_{i+1}c_i}{a_i}}A_i \mathbf{x}_i,\]
for a constant $c_i$ and a complex orthogonal matrix $A_i$. Using
that $d(z_i)=1, d(x_{i,j})=2$, and that the $z_i, x_{i,j}$ have
normalized $L^2$-norm approximately equal to 1 on $B_i$, we have
$c_i\to 2^{-1}$ and $\sqrt{a_i^{-1}a_{i+1}c_i} \to 2^{-2}$ as
$i\to\infty$. It follows that $a_i^{-1}a_{i+1} \to 2^{-3}$.

The cases $n=4$ and $n>4$ can be treated in an almost identical
way, except in both cases we have $\dim I_2 = 5$, so we do not have to
worry about the function $z^2$. In addition in the equation analogous to
\eqref{eq:reducedeq}, the degree of $\tilde{f}_i$ is at most cubic when
$n=4$, and at most quadratic when $n > 4$. We leave the details to the reader. 

\section{Decay of the K\"ahler potential}\label{sec:decay}
In this section, we let $\omega = c^2\omega_0$ be a scaled
down copy of the reference Calabi-Yau metric on $\mathbf{C}^n$,
and we denote by $o\in
(\mathbf{C}^n,\omega)$ the origin. For a given $\epsilon > 0$ we will
assume that $c$
is sufficiently small so that $d_{GH}(B(o, \epsilon^{-1}),
B(0, \epsilon^{-1})) < \epsilon$, where $0\in C(Y)$ is the origin in
the cone $C(Y) = \mathbf{C}\times A_1$. Since $C(Y)$ is the tangent
cone at infinity, this condition is equivalent to the Gromov-Hausdorff
distance $d_{GH}(B(o,1), B(0,1))$ being sufficiently small. 
As in Section~\ref{sec:ref} we have
an embedding $F_c : \mathbf{C}^n\to \mathbf{C}^{n+1}$ using the functions $cz,
c^{\frac{n-1}{n-2}}x_i$, with image given by the equation
\[ c^{\frac{n}{n-2}}z + x_1^2 +\ldots +x_n^2 = 0. \]
In addition the maps $F_c$ on $B(o,1)$ converge in the
Gromov-Hausdorff sense to the standard embedding
$B(0,1) \to \mathbf{C}^{n+1}$ of the unit ball in $C(Y)$ as $c\to 0$.

The main technical result of the paper is the following. 
\begin{prop}\label{prop:decayest}
  There are $\lambda_0, \alpha$ such that if $\lambda <
  \lambda_0$, and $\epsilon$
  is sufficiently small (depending on $\lambda$), then
  we have the following. 
  Suppose that $c$ above is small enough so that
  $d_{GH}(B(o, \epsilon^{-1}), B(0, \epsilon^{-1})) < \epsilon$, and 
  we have a smooth
  function $u$ on $B(o,1)$ satisfying $\sup_{B(o,1)}|u| < \epsilon$ and
  \[ (\omega + \ddb u)^n = \omega^n. \]
  Then we can find a constant $\beta$, an automorphism $g$ of
  $\mathbf{C}^n$ fixing the origin $o$, and
  a smooth function $u'$ on $B(o,1)$  satisfying
  \begin{enumerate}
    \item $\beta g^*(\omega + \ddb u) = \omega + \ddb u'$,
    \item $(\beta g^*\omega)^n = \omega^n$,
    \item $\sup_{B(o, \lambda)} |u'| \leq \lambda^{2+\alpha}
      \sup_{B(o,1)} |u|$. 
  \end{enumerate}
\end{prop}

In this work in the property (3) the constant $\lambda^2$ would
suffice, however we expect that the result has other applications that
need the better constant $\lambda^{2+\alpha}$. 
The strategy of the proof is to show that once $\epsilon$
is sufficiently small, the function $u$ is close to a harmonic function on
$C(Y)$. We then modify $u$ by subtracting a pluriharmonic function,
and applying an automorphism to get a function $u'$ with faster than
quadratic decay using Lemma~\ref{lem:vbeta}. The main difficulty
is to control the behavior of $u$ near the singular set of $C(Y)$. We
will achieve this by using the maximum principle with a suitable
barrier function in order to show that if $u$ concentrates near the
singular set, then $u$ decays rapidly when passing from $B(o,1)$ to
$B(o,1/2)$. Note that much of the argument works in more
general settings than what we are considering. The one place where we will use the
fairly explicit form of the reference metric $\omega_0$ on $\mathbf{C}^n$ is when
we need to control the action of the automorphisms on the K\"ahler
potential of $\omega$ near the singular set, using
Proposition~\ref{prop:vlingrowth}. 

To begin, we note that
on any compact set away from the singular set of $C(Y)$ we can apply
the small perturbation result of Savin~\cite{Sav07} to get regularity
of $|u|$ once $\epsilon$ is sufficiently small. For $\theta >
0$ we define the set $N_\theta$ to be the $\theta$-neighborhood of the
singular set under the Gromov-Hausdorff approximation:
\[ N_\theta = \{ (z,\mathbf{x})\,|\, 2^{1/2}|x|^{\frac{n-2}{n-1}} < \theta\}, \]
where we note that on $C(Y)$ the function $2|x|^{2\frac{n-2}{n-1}}$ is the distance
squared from the singular set under the Ricci flat cone metric. 
\begin{lem}\label{lem:Savin}
  Let $\theta > 0$. There exist $C_k> 0$ depending on $\theta$
  with the following property. 
  If, in the setting of Proposition~\ref{prop:decayest}, $\epsilon$ is
  sufficiently small (depending on $\theta$), then on
  $B(o,{1-\theta}) \setminus N_\theta$ we have
  \[  |u|_{C^k(B(o,1-\theta)\setminus N_\theta)} &< C_k\sup_{B(o,1)}|u|. \]
\end{lem}
\begin{proof}
  Note that by the Cheeger-Colding theory~\cite{CC97} and Anderson's
  epsilon regularity result~\cite{An90}, we can bound the harmonic
  radius of $\omega$ on
  $B(o,1-\theta)\setminus N_\theta$ once $\epsilon$ is sufficiently
  small. Then for any $\delta > 0$, if $\sup_{B(o,1)}|u|$ is sufficiently small, then we can
  apply Savin's result~\cite{Sav07} in harmonic coordinates to obtain
  $|u|_{C^3(B(o,1-\theta/2))} < \delta$. We have the equation
  \[ (\omega + \ddb u)^n = \omega^n, \]
  which we can write as
  \[ \label{eq:omegaueqn}
    \left[n\omega^{n-1} + \binom{n}{2}\omega^{n-2}\wedge\ddb u +\ldots +
      (\ddb u)^{n-1}\right]  \wedge \ddb u =
    0. \]
  If $\delta$ is sufficiently small (so that $\ddb u$ is small), then we
  can view this as a uniformly elliptic homogeneous equation for $u$
  with $C^\alpha$ coefficients (the expression in square brackets
  determining the coefficients of the equation). The Schauder
  estimates imply that on $B(o,1-3\theta/4)$ we have the bound
  $|u|_{C^{2,\alpha}} \leq C_2 \sup_{B(o,1)} |u|$. We can bootstrap
  this estimate to obtain the required higher order bounds. 
\end{proof}

We will apply this result in the following form several times.
\begin{lem}\label{lem:hlimit}
  Suppose that we have a sequence of functions $u_i$ as in
  Proposition~\ref{prop:decayest} such that $\epsilon_i \to
  0$. Let $v_i = (\sup_{B(o_i,1)}|u_i|)^{-1}u_i$. Then there is a
  (possibly vanishing)
  bounded harmonic function $h$ on $B(0,1)\subset C(Y)$ such that under the
  Gromov-Hausdorff convergence, for a subsequence,
 we have $v_i \to h$ in $C^\infty$
  uniformly on $B(o_i,1-\theta)\setminus N_\theta$ for any $\theta >
  0$. 
\end{lem}
\begin{proof}
  By Lemma~\ref{lem:Savin}, for any given $\theta > 0$ the functions
  $v_i$ satisfy uniform $C^\infty$ bounds on $B(o_i,
  1-\theta)\setminus N_\theta$ for
  sufficiently large $i$, and they are all bounded by 1 on
  $B(o_i,1)$. 
  By a diagonal argument we can find a
  subsequence such that $v_i \to h$ for a function $h$ on
  $B(0,1)\subset C(Y)$, with the convergence taking place in
  $C^\infty$ on any $B(o_i, 1-\theta)\setminus N_\theta$.

  To see that $h$ is harmonic, note that just like the equation
  \eqref{eq:omegaueqn}, $v_i$ satisfies
  \[ \left[n\omega_i^{n-1} + \binom{n}{2}\omega_i^{n-2}\wedge \ddb u_i
      +\ldots + (\ddb u_i)^{n-2}\right] \wedge 
    \ddb v_i=0. \]
  By Lemma~\ref{lem:Savin}, on $B(o_i, 1-\theta)\setminus N_\theta$ we
  have $u_i\to 0$ in $C^\infty$, and so passing to the limit in these
  equations we get that $h$ is harmonic on  $B(0,1-\theta)\setminus
  N_\theta$. This holds for any $\theta$, and so $h$ is harmonic on
  the regular part of $B(0,1)$. In addition $h$ is bounded, so it is
  harmonic in a weak sense across the singular set too. 
\end{proof}

\subsection{Construction of a barrier function}
We suppose that we are in the setting of
Proposition~\ref{prop:decayest}. The following provides the barrier
function used in the maximum principle argument below. 
\begin{prop}\label{prop:barrier}
  There is a constant $D > 0$ with the following property. Let $\theta
  > 0$. There is a constant $C_\theta > 0$ such that if
  $\epsilon$ is sufficiently small (depending
  on $\theta$), then there is a smooth real function $v$ on $B(o,1)$
  satisfying the following properties: 
  \begin{enumerate}
  \item $|\partial\bar\partial v|_{\omega} < C_\theta$ on
    $B(o,1-\theta/2)$, 
    \item $v(q) > D^{-1}\theta^{-1/2}$ whenever $q\in
      N_\theta \cap \partial B(o,1-\theta)$.
    \item $v > D^{-1}$ on $B(o,1)$, and $v < D$ on $B(o,1/2)$.
    \item On $B(o,1-\theta/2)$ the function $v$ satisfies the
      differential inequality
      \[ \sum_i \mu_i + \mu_{max} < 0, \]
      where the $\mu_i$ are the eigenvalues of $\ddb v$ relative to
      $\omega$, and $\mu_{\max}$ is the largest eigenvalue.  
  \end{enumerate}
\end{prop}

\begin{proof}
Let $(z,0)$ be a point in the singular set of $C(Y)$, with $|z|=1$,
and let $q\in B(o,2)$ be within $\epsilon$ of $(z,0)$ under the
Gromov-Hausdorff approximation. Using that $C(Y)$ is
a cone also when centered at $(z,0)$, once $\epsilon$ is sufficiently
small we can 
apply \cite[Proposition 3.1]{LSz18} to find a good K\"ahler
potential $\phi$ for $\omega$ on $B(q,3)$, in the sense that we have
$\omega = \ddb \phi$, while also
\[ \label{eq:phid} | \phi - d(q,\cdot)^2/2| < \Psi(\epsilon) \]
on $B(q,3)$. By adding a constant of order $\Psi(\epsilon)$
we can assume that $\phi > 0$. We
have $\Delta \phi = n$ (using the complex Laplacian),
and from the Cheeger-Colding
estimate~\cite{CC96} together with the Cheng-Yau gradient
estimate~\cite{CY75} we have
\[ \int_{B(q,2)} \Big| |\nabla^{1,0}\phi|^2 - \phi\Big|^2 <
  \Psi(\epsilon), \]
where we emphasize that we are taking the $(1,0)$-part of the
derivative of $\phi$. At the same time from the Bochner formula
(computing in normal coordinates)
\[ \Delta |\nabla^{1,0}\phi|^2 = (\phi_i \phi_{\bar i})_{j\bar j} =
  \phi_{ij}\overline{\phi_{ij}} + \phi_{i\bar j}\overline{\phi_{i\bar
      j}} \geq n, \]
since $\phi_{i\bar j} = \omega_{i\bar j}$. So
\[ \Delta \Big( |\nabla^{1,0}\phi|^2 - \phi\Big) \geq 0, \]
and the mean value inequality implies that
\[ |\nabla^{1,0}\phi|^2 \leq \phi + \Psi(\epsilon). \]

Let us now consider the function $\phi^{-3/4}$. We have
\[ \ddb \phi^{-3/4} = -\frac{3}{4}\phi^{-7/4} \ddb\phi + \frac{21}{16}
  \phi^{-11/4} \sqrt{-1}\partial\phi\wedge\bar\partial\phi. \]
Fix $y\in B(o,1)$. We can choose orthonormal coordinates for $\omega$
at $y$ such that $\partial \phi = \phi_1 dz^1$, and $\ddb\phi$
is the identity matrix. By the estimate above, we have $|\phi_1|^2\leq
\phi + \Psi(\epsilon)$. The eigenvalues of $\ddb \phi^{-3/4}$
therefore satisfy
\[ -\frac{3}{4}\phi^{-7/4} \leq \mu_1 &\leq -\frac{3}{4}\phi^{-7/4} +
  \frac{21}{16}\phi^{-11/4}(\phi + 
  \Psi(\epsilon)),  \\
  \mu_2,\ldots, \mu_n &= -\frac{3}{4}\phi^{-7/4}. \]
The maximum eigenvalue is necessarily $\mu_1$, and so
\[ \label{eq:supersol}
  \sum \mu_i + \mu_{max} = 2\mu_1 + (n-1)\mu_2 &\leq 2\mu_1 + 2\mu_2 \\
  &\leq -3\phi^{-7/4} +
  \frac{21}{8}\phi^{-11/4}(\phi + \Psi(\epsilon)) \\
  &= \phi^{-11/4}\left( -\frac{3}{8}\phi + \frac{21}{8}\Psi(\epsilon)\right). \]
As long as $\phi > 7\Psi(\epsilon)$, we obtain
\[ \label{eq:diffineq1} \sum \mu_i + \mu_{max}  < 0. \]
In particular by \eqref{eq:phid} this holds if $y\in B(o,1-\theta)$, and $\epsilon$ is
sufficiently small (depending on $\theta$). 

We also have $\phi > \theta^2/4$ on $B(o,1-\theta)$ once $\epsilon$ is
sufficiently small, and so from the
bounds for the eigenvalues we have $|\ddb\phi^{-3/4}| < C_\theta$ on
$B(o,1-\theta)$ (where $C_\theta$ is of order $\theta^{-7/2}$,
although we do not need this). On $B(o,1/2)$ we have $\phi > 1/10$ once
$\epsilon$ is sufficiently small, and this leads to an upper bound
$\phi^{-3/4} < 10^{3/4}$ on $B(o,1/2)$. At the same time $\phi < 1$ on
$B(o,1)$ for sufficiently small $\epsilon$, and so $\phi^{-3/4} > 1$
on $B(o,1)$. 

Note that once $\epsilon$ is sufficiently small, we have $\phi <
16\theta^2$ on $B(q,4\theta)$, and so 
\[ \phi^{-3/4} > 16^{-3/4}\theta^{-3/2} \text{ on } B(q, 4\theta). \]
This means $\phi^{-3/4}$ is large near $q$,
but we want a function that is large at all points in
$N_\theta\cap \partial B(o,1-\theta)$. To achieve this, we define $v$
to be an average of functions $\phi^{-3/4}$ constructed for different
points in $S^1\times \{0\} \in \mathbf{C}\times
A_1$. For a given $\theta > 0$, we pick
$z_1,\ldots, z_K$ on the unit circle, and $q_i\in B(o,2)$ which are
$\epsilon$-close to $(z_i, 0)\in C(Y)$ under our Gromov-Hausdorff
approximation, so that the $4\theta$-balls $B(q_i, 4\theta)$ cover
$N_\theta\cap \partial B(o, 1-\theta)$. For sufficiently small
$\epsilon$ we can achieve this with $K < c\theta^{-1}$, for a uniform
$c$. We consider the functions $\phi_i^{-3/4}$ constructed as
above, based at the points $q_i$, and define
\[ v = \frac{1}{K} \sum_{i=1}^K \phi_i^{-3/4}. \]
Then $v$ satisfies the required properties:
\begin{enumerate}
\item $|\ddb v|_\omega < C_\theta$,  since we are taking an average of
  functions that satisfy this estimate. 
\item If $q\in N_\theta\cap B(o,1-\theta)$, then by assumption, there
  is a $q_i$ such that $q\in B(q_i, 4\theta)$. It follows that 
\[ v(q) > \frac{1}{K}\phi_i^{-3/4}(q) > c^{-1}\theta\cdot
  16^{-3/4}\theta^{-3/2} = c^{-1}16^{-3/4} \theta^{-1/2}, \]
which gives the required lower bound. 
\item We have $\phi_i^{-3/4} > 1$ on $B(o,1)$ for all $i$, so $v$
  satisfies the same estimate. Similarly, $\phi_i^{-3/4} < 10^{3/4}$
  on $B(o,1/2)$, and so $v$ satisfies the same.
\item Each $\phi_i^{-3/4}$ satisfies the required differential
  inequality \eqref{eq:diffineq1} on $B(o,1-\theta/2)$, and the expression $\sum_i \mu_i +
  \mu_{max}$ is convex on the space of Hermitian matrices. Therefore
  $v$ satisfies the same differential inequality. 
\end{enumerate}
\end{proof}

Note that this result can easily be generalized to other cones of the form $C(Y) =
  \mathbf{C}^k\times C(Y')$, where $Y'$ has an isolated
  singularity, but we have crucially used that all singular
  points in $C(Y)$ can be taken to be a vertex of $C(Y)$. We still
  expect that with some additional work a similar barrier function can
  be constructed for more general cones.

We now use the maximum principle to obtain the following important
decay property.
\begin{prop}\label{prop:maxdecay}
  There is a constant $C > 0$ with the following property. Let $A
  > 10$. There exists $\theta > 0$ depending on $A$, such that if in
  the setting of Proposition~\ref{prop:decayest}
  $\epsilon$ is sufficiently small (depending on $A,\theta$),
  and
  \[ \sup_{B(o,1)} |u| \leq A \sup_{B(o,1)\setminus N_\theta} |u|, \]
  then
  \[ \sup_{B(o,1/2)}|u| \leq C\sup_{B(o,1)\setminus N_\theta} |u|. \]
  Note that $C$ does not depend on $A,\theta$. 
\end{prop}
We first need the following lemma.
\begin{lem}\label{lem:ineq}
  Suppose that $\mu_i > -1$ are constants for $i=1,\ldots,n$ such that
  \[\prod_{i=1}^n (1+\mu_i) = 1. \]
  There is a $\delta_0 > 0$ depending
  only on $n$, such that if $\mu_i < \delta_0$ for all $i$, then
  \[ \sum_{i=1}^n \mu_i + \mu_{max} \geq 0, \]
  while if $\mu_i > -\delta_0$ for all $i$, then 
  \[ \sum_{i=1}^n \mu_i + \mu_{min} \leq 0. \]
  Here $\mu_{max}, \mu_{min}$ are the largest and smallest
  of the $\mu_i$. 
\end{lem}
\begin{proof}
  Suppose that $\mu_{max} = \delta < \delta_0$. For all $j$ we have
  \[ 1+\mu_j = \frac{1}{\prod_{i\ne j} (1+\mu_i)} \geq
    (1+\delta)^{-(n-1)}. \]
  We can choose $\delta_0$ so that if $0 < \delta < \delta_0$, then 
  \[ (1+\delta)^{-(n-1)} \geq 1-n\delta, \]
  so we find $\mu_j \geq -n\delta$ for all $j$. Then if $\delta < 1$,
  we have
  \[ 1 = \prod_{i=1}^n (1+\mu_i) \leq 1 + \sum_{i=1}^n \mu_i +
    C_n\delta^2, \]
  for a constant $C_n$ depending only on $n$, since the remaining
  terms are all at least quadratic in the $\mu_i$. It follows that
  \[ \sum_{i=1}^n \mu_i + \mu_{max} \geq -C_n\delta^2 + \delta. \]
  Finally, if $\delta$ is sufficiently small, then $C_n\delta^2 \leq
  \delta$. The argument for the second statement is completely
  analogous. 
\end{proof}

\begin{proof}[Proof of Proposition~\ref{prop:maxdecay}]
  Let us choose $\theta  = A^{-2}$, and assume $\epsilon$ is
  sufficiently small to apply Proposition~\ref{prop:barrier}. Let
  $\Lambda > 0$ satisfy
  \[ \sup_{B(o,1)\setminus N_\theta} |u| = \Lambda^{-1} D^{-1}, \]
  and suppose that
  \[ \sup_{B(o,1)} |u| &\leq A\sup_{B(0,1)\setminus N_\theta} |u| = \Lambda^{-1}
    D^{-1}\theta^{-1/2}. \]
  In terms of $v$ given by Proposition~\ref{prop:barrier}, 
  set $\tilde{v} = \Lambda^{-1}v$, so that
  $\tilde{v} > u$ on $\partial B(o, 1-\theta)$ by properties
  (2) and (3). We assume in addition
  that $\epsilon$ is sufficiently small (and so $\sup_{B(o,1)}|u|$ is
  sufficiently small) so that $\Lambda^{-1}C_\theta < \delta_0$ for
  the $\delta_0$ from Lemma~\ref{lem:ineq}. 
  
  \bigskip
  \noindent{\bf Claim.} $\tilde{v} > u$ on $B(o,1-\theta)$. \\
  \noindent{\em Proof of Claim.} If this were not the case,
  then setting
  \[ t_0 = \inf\{ t > 0\,:\, \tilde{v} +t > u \text{ on }
    B(o,1-\theta)\}, \]
  the graph of $\tilde{v} + t_0$ will lie above the graph of $u$, and
  the two graphs will touch at a point $q\in B(o,1-\theta)$. At $q$ we
  must have $\ddb u(q) \leq \ddb \tilde{v}(q)$. In orthonormal
  coordinates at $q$ we have $\ddb \tilde{v}(q) \leq
  \Lambda^{-1}C_\theta \mathrm{Id}$ by property (1) in
  Proposition~\ref{prop:barrier}, and so the eigenvalues $\mu_i$ of $\ddb
  u(q)$ are bounded above by $\Lambda^{-1}C_\theta$.  Since
  $\Lambda^{-1}C_\theta < \delta_0$,  Lemma~\ref{lem:ineq} implies
  that $\sum \mu_i + \mu_{max} \geq 0$, but this contradicts property
  (4) in Proposition~\ref{prop:barrier}, since $\ddb u(q) \leq
  \ddb\tilde{v}(q)$. This proves the claim.

  \bigskip
  Using the claim, it follows from property (3) that
  \[ \sup_{B(o,1/2)} u \leq \Lambda^{-1}D = D^2\sup_{B(o,1)\setminus
      N_\theta}|u|, \]
  which gives the required upper bound for $u$.

  The lower bound for $u$ is proved similarly, just comparing with the
  function $-\Lambda^{-1}v$ instead, and using the second statement in
  Lemma~\ref{lem:ineq}. 
  \end{proof}

\subsection{Proof of Proposition~\ref{prop:decayest}}
We prove Proposition~\ref{prop:decayest} by contradiction. Let us
suppose that we have a sequence $u_i$ as in the statement of the
proposition, on balls $B(o_i,1)$, with corresponding constants
$\epsilon_i\to 0$, such that the conclusion of the proposition
fails. We will show that along a subsequence,
for sufficiently large $i$ we can find $\beta_i,g_i, u'_i$ satisfying the required
properties, giving a contradiction. 

\bigskip
\noindent{\bf Step 1.} 
Let us write $\kappa_i = \sup_{B(o_i,1)} |u_i| \to 0$. By
Lemma~\ref{lem:hlimit} we have a harmonic function $h$ on
$B(0,1)\in C(Y)$ such that, after choosing a subsequence,
$\kappa_i^{-1} u_i\to h$ in $C^\infty$ uniformly on $B(o_i,1-\theta)\setminus
N_\theta$ for any $\theta > 0$. Note that we have $|h|\leq 1$, and it
is possible that $h=0$. Let us write $h=h^{\leq 2} + h^{>2}$ for the
decomposition of $h$ into pieces with at most quadratic and faster
than quadratic growth. In addition we decompose $h^{\leq 2} = h_{ph} +
h_{aut}$, where $h_{ph}$ is pluriharmonic, and $h_{aut}$ is in the
span of the functions of type (2) and (3) in
Lemma~\ref{lem:CYharmonic}. From Lemma~\ref{lem:vbeta} we obtain a real
holomorphic vector field $V$ on $\mathbf{C}^{n+1}$ preserving the
hypersurfaces
\[ az + x_1^2 + \ldots + x_n^2=0, \]
and a constant $\beta$, such that $L_V\Omega = n\beta\Omega$, and at the same
time on $C(Y)\subset\mathbf{C}^4$ we have
\[ V(|z|^2 + |x|^{2\frac{n-1}{n-2}}) - \beta(|z|^2+|x|^{2\frac{n-1}{n-2}}) = h_{aut}. \]
We define the automorphism $g_i = \exp(\kappa_i V)$ on $\mathbf{C}^{n+1}$,
and the constants $\beta_i = \kappa_i\beta$.

For any $\theta >
0$, the spaces $B(o_i, 1)\setminus N_\theta$ converge smoothly to
$B(0,1)\setminus N_\theta$ inside $\mathbf{C}^{n+1}$, and for sufficiently
large $i$ we can use the nearest point projection to identify
them. On
$B(o_i, 1)$ we have K\"ahler potentials $\phi_i$ for $\omega_i$, which
converge smoothly to $|z|^2 + |x|^{2\frac{n-1}{n-2}}$ as $i\to \infty$, uniformly on
$B(o_i,1)\setminus N_\theta$ for any $\theta > 0$. It follows from
this, that
\[ \sup_{B(o_i,1)\setminus N_\theta} \Big| \kappa_i V\phi_i -
  \beta_i\phi_i - \kappa_i h_{aut}\Big| \leq \kappa_i
  \Psi(i^{-1}\,|\,\theta), \]
and so we also have
\[ \sup_{B(o_i,1)\setminus N_\theta} \Big| e^{-\beta_i}g_i^*\phi_i -
  \phi_i - \kappa_i h_{aut}\Big| \leq 
  \kappa_i\Psi(i^{-1}\,|\,\theta). \]
Here $\Psi(i^{-1}\,|\,\theta)$ denotes a function, which for fixed
$\theta$ converges to zero as $i\to\infty$. 
At the same time, using Proposition~\ref{prop:vlingrowth}, we have a
uniform bound $|V|_{\omega_i} < C$ on $B(o_i,1)$, since this ball is a
ball centered at the origin in our reference metric
$(\mathbf{C}^n,\omega_0)$ scaled down to unit size.  Together with the
uniform gradient bound for $\phi_i$ (since $\Delta_{\omega_i}\phi_i =
n$), this implies
\[ \sup_{B(o_i,1)} \Big| e^{-\beta_i}g_i^*\phi_i - \phi_i \Big| \leq
  C\kappa_i. \] 

To deal with $h_{ph}$ note that $h_{ph}$ is in the span of the real and
imaginary parts of $1,z,z^2,x_i$, and so $h_{ph}$ also defines
a pluriharmonic function $h_{ph,i}$ on $B(o_i,1)$ for all $i$ under the
embeddings into $\mathbf{C}^{n+1}$. For any $\theta > 0$ we have
\[ \sup_{B(o_i,1)\setminus N_\theta} \Big| h_{ph,i} - h_{ph} \Big| <
  \Psi(i^{-1}\,|\,\theta), \]
where as above, we are using the nearest point projection on
$B(o_i,1)\setminus N_\theta$ for sufficiently large $i$ to view
$h_{ph}$ as a function on $B(o_i,1)$. We also clearly have
\[ \sup_{B(o_i,1)} | h_{ph,i}| \leq
  C. \]

We can now define
\[ u_i' = u_i - \kappa_ih_{ph,i} - (e^{-\beta_i}g_i^*\phi_i -
  \phi_i). \]
By the construction we have
\[ \label{eq:bgeq1}
  \omega_i + \ddb u_i = e^{-\beta_i}g_i^*\omega_i + \ddb u_i', \]
and $e^{-\beta_i}g_i^*\omega_i$ has the same volume form as
$\omega_i$. By the estimates above, we have
\[ \sup_{B(o_i,1)\setminus N_\theta} |u_i' - h_i^{> 2}|
  \leq \Psi(i^{-1}\,|\,\theta) \kappa_i, \]
and also 
\[ \sup_{B(o_i,1)} |u_i'| \leq C\kappa_i. \]
Letting $\theta\to 0$, we find that 
\[ \Vert u_i' - h_i^{> 2} \Vert_{L^2(B(o_i,1))} 
  \leq \Psi(i^{-1}) \kappa_i, \]
where near the singular set
we use a Gromov-Hausdorff approximation to view $h_i^{>2}$ as a
function on $B(o_i,1)$.

Since $h_i^{>2}$ has faster than quadratic growth, there exists an
$\alpha > 0$ (depending only on the cone $C(Y)$), such that
\[ \Vert h_i^{>2}\Vert_{B(0,\lambda)} \leq \lambda^{2+2\alpha} \Vert
  h_i^{>2}\Vert_{B(0,1)}. \]
Here, and below, for any ball $B$ we define
\[ \Vert f\Vert_B = \left( \mathrm{Vol}(B)^{-1}\int_B
    |f|^2\right)^{1/2}\]
to be the $L^2$-norm normalized by the volume of the ball $B$.

We therefore have
\[ \Vert h_i^{>2}\Vert_{B(0,\lambda)} &\leq  C\lambda^{2+2\alpha}
  \kappa_i, \]
while also
\[ \Vert h_i^{>2}\Vert_{B(0, \lambda)} &\geq \Vert
  u_i'\Vert_{B(o_i,\lambda)} - \Vert h_i^{>2} -
  u_i'\Vert_{B(o_i,\lambda)} \\
  &\geq \Vert
  u_i'\Vert_{B(o_i,\lambda)} - C_\lambda \Vert h_i^{>2} -
  u_i'\Vert_{B(o_i,1)} \\
  &\geq \Vert
  u_i'\Vert_{B(o_i,\lambda)} - C_\lambda \Psi(i^{-1}) \kappa_i, \]
  for a constant $C_\lambda$ depending on $\lambda$. Combining these,
  we get
  \[ \Vert u_i' \Vert_{B(o_i,\lambda)} \leq (C\lambda^{2+2\alpha} +
    C_\lambda\Psi(i^{-1}))\kappa_i. \]
  Once $i$ is sufficiently large (depending on $\lambda$), we get
  \[ \Vert u_i' \Vert_{B(o_i,\lambda)} \leq 2C\lambda^{2+2\alpha}\kappa_i.\]

  \bigskip
  \noindent{\bf Step 2.} Let us apply the construction in Step 1 to
  $8\lambda$ instead of $\lambda$, and let us scale the balls $B(o_i,
  8\lambda)$ up to unit size.  We denote the origins of the
  scaled up balls by $o_i'$, and also let $\omega_i' =
  (8\lambda)^{-2}\beta_i g_i^*\omega_i$.  Letting $U_i' = (8\lambda)^{-2}u_i'$, on
  $B(o_i',1)$ we have
  \[ (\omega_i' + \ddb U_i')^3 = \omega_i'^3, \]
  and for fixed $\lambda$, as $i\to \infty$,
  the metrics $\omega'_i$ on $\mathbf{C}^n$ satisfy the same assumptions as
  $\omega$ in the statement of Proposition~\ref{prop:decayest} for
  arbitrarily small $\epsilon$. By Step 1, (replacing $C$ by a
  larger constant if necessary) we have
  \[ \label{eq:U'bounds}
    \sup_{B(o_i',1)} |U_i'| &\leq C\lambda^{-2} \kappa_i, \\
    \Vert U_i'\Vert_{B(o_i', 1)} &\leq C\lambda^{2\alpha} \kappa_i. \]
  By Lemma~\ref{lem:hlimit} we have a harmonic function $H$ on
  $B(0,1)\subset C(Y)$, such that after choosing a subsequence
  \[ \frac{U_i'}{\sup_{B(o_i',1)} |U_i'|} \to H, \]
  the convergence being in $C^\infty$ uniformly on
  $B(o_i',3/4)\setminus N_\theta$ for any $\theta > 0$. There are two
  cases:
  \begin{itemize}
    \item Suppose that $H\not= 0$. Then by the $L^\infty$ bound for
      harmonic functions on $C(Y)$, we have
      \[ \sup_{B(0,1/2)} |H| \leq C \Vert H\Vert_{B(0,1)}. \]
      It follows that for any $\theta > 0$, once $i$ is sufficiently
      large, we have
      \[ \frac{ \sup_{B(o_i',1/2)\setminus N_\theta}
          |U_i'|}{\sup_{B(o_i',1)} |U_i'|} &\leq 2\sup_{B(0,1/2)}|H|
        \\
        &\leq 2C\Vert H\Vert _{B(0,1)} \\
        &\leq 4C \frac{ \Vert
          U_i'\Vert_{B(o_i',1)}}{\sup_{B(o_i',1)}|U_i'|}, \]
      and so using \eqref{eq:U'bounds} we have
      \[ \sup_{B(o_i',1/2)\setminus N_\theta} |U_i'| \leq 4C^2\lambda^{2\alpha}
        \kappa_i. \]
      \item Suppose that $H=0$. Then for any $\theta > 0$ we
        have
        \[ \sup_{B(o_i',1/2)\setminus N_\theta} |U_i'| \leq \lambda^{2+2\alpha}
          \sup_{B(o_i',1)} |U_i'| \leq C\lambda^{2\alpha}\kappa_i, \]
        once $i$ is sufficiently large.
      \end{itemize}
      In either case, applying Proposition~\ref{prop:maxdecay} to $B(o_i',1/2)$, we
      can choose $\theta > 0$ depending on $\lambda$, such that for
      sufficiently large $i$ we have
      \[ \sup_{B(o_i', 1/4)} |U_i'| \leq C'\lambda^{2\alpha} \kappa_i, \]
      for a constant $C'$ depending only on $C(Y)$. After rescaling,
      we get the bound $\sup_{B(o_i,2\lambda)} |u_i'| \leq
      64C'\lambda^{2+2\alpha}\kappa_i$, and from \eqref{eq:bgeq1} we
      have
      \[  e^{\beta_i}(g_i^{-1})^*(\omega_i + \ddb u_i) = \omega_i
        +\ddb\left( e^{\beta_i}(g_i^{-1})^*u_i'\right). \]
      Letting $u_i'' = e^{\beta_i}(g_i^{-1})^*u_i'$ we have
      \[ \sup_{B(o_i, \lambda)} |u_i''| \leq
        100C'\lambda^{2+2\alpha}\kappa_i \]
      as long as $g_i^{-1}(B(o_i, \lambda))\subset B(o_i,2\lambda)$,
      and $e^{\beta_i} < 3/2$. Both of these estimates will hold once
      $i$ is sufficiently large (depending on $\lambda$). Finally we
      just need to ensure that $\lambda$ is sufficiently small so that
      $100C'\lambda^{2+2\alpha} < \lambda^{2+\alpha}$.

      \subsection{Proof of Theorem~\ref{thm:main}}
      We now  prove the main result,
      Theorem~\ref{thm:main}. Suppose that we have a Calabi-Yau
      metric $(X',\eta)$ with $X'$ biholomorphic to $\mathbf{C}^n$, and
      with tangent cone $\mathbf{C}\times
     A_1$ at infinity. Let us write $(X,\omega_0)$ for our reference
     metric on $\mathbf{C}^n$ discussed in Section~\ref{sec:ref}. 
    We have the origin $o\in X$, and
     using Lemma~\ref{lem:o} we also have a distinguished basepoint
      $o'\in X'$. 
      By the discussion in Section~\ref{sec:ref} we have
      embeddings $F_i:X\to\mathbf{C}^{n+1}$ such that $F_i(o)=0$, and the
      image $F_i(X)$ has equation
      \[ a_iz + x_1^2+\ldots + x_n^2=0, \]
      where $a_i = 2^{-in/(n-2)}$. At the same time using
      Proposition~\ref{prop:goodembed} we have
      embeddings $F_i':X'\to\mathbf{C}^{n+1}$
      such that $F_i'(o)=0$, and the image $F_i'(X')$
      satisfies the equation
      \[   a_i'z + x_1^2+ \ldots +x_n^2 &= 0, \]
      with $a_i' > 0$. In addition the unit balls for
      the scaled down metrics $2^{-2i}\omega_0, 2^{-2i}\eta$ converge
      in the Gromov-Hausdorff sense
      to the unit ball $B(0,1)\subset C(Y)$, and the maps
      $F_i, F_i'$ converge to the standard embedding $B(0,1)\subset
      \mathbf{C}^{n+1}$ as $i\to\infty$. 

      We need to find suitable $j(i)$ such that
      after a further scaling by a bounded factor, the images
      $F_{j(i)}'(X')$ satisfy the same equations as $F_i(X)$. 
      Since $a_i/a_{i+1}= 2^{n/(n-2)}$ and $a_i'/a_{i+1}'\to 2^{n/(n-2)}$, for all
      sufficiently large $i$ we can find $j(i)$ such that
      \[ C_n^{-1} a_{j(i)}' < a_i < C_n a_{j(i)}', \]
      for a dimensional constant $C_n$,   and $j(i)\to\infty$ as
      $i\to\infty$. Let us compose
      $F_{j(i)}'$ with an automorphism $g_i$ of $\mathbf{C}^{n+1}$ of the form
      $(z,\mathbf{x})\mapsto (c_iz, c_i^{\frac{n-1}{n-2}}\mathbf{x})$ for $c_i >
      0$. Then $g_i\circ F_{j(i)}'(X')$ satisfies the equation
      \[ c_i^{\frac{n}{n-2}}a_{j(i)}'z + x_1^2 + \ldots + x_n^2 = 0, \]
      so we can choose $c_i\in (C_n^{-1}, C_n)$ so that $g_i\circ
      F_{j(i)}'(X')$ satisfies the same equation as $F_i(X)$. The
      balls $B_\eta(o', c_i^{-1}2^{-j(i)})$ with the scaled metrics
      $\eta_i = c_i^{-2}2^{-2j(i)}\eta$ still converge to the unit ball
      $B(0,1)\subset C(Y)$, and on these balls the maps
      $g_i\circ F'_{j(i)}(X')$ still converge to the standard
      embedding of $B(0,1)$ into
      $\mathbf{C}^{n+1}$ as $i\to\infty$. Moreover the volume form of $\eta_i$ is the
      pullback of $\sqrt{-1}^{n^2}\Omega\wedge\bar\Omega$ under $g_i\circ
      F_{j(i)}'$. 

      Let us write $\omega_i = 2^{-2i}\omega_0$, and $B_i$ for the unit
      ball around $o$ with the metric $\omega_i$. Similarly let $B_i'$
      be the unit ball around $o'$ with the metric $\eta_i$. We have
      biholomorphisms $\Phi_i : X \to X'$ defined by
      \[ \Phi_i = (g_i\circ F_{j(i)}')^{-1}\circ F_i, \]
      which satisfy $\Phi_i(o)=o'$, and $\Phi_i^*(\eta^n) =
      \omega_0^n$. We claim that on the ball $B_i$ we have
      \[ \label{eq:d2} | \Phi_i^*d_{\eta_i}(o', \cdot) - d_{\omega_i}(o,\cdot)| <
        \Psi(i^{-1}). \]
      For this, let $x\in B_i$, and $x' = \Phi_i(x)$. By the
      construction and Proposition~\ref{prop:goodembed}
      we have $\Psi(i^{-1})$-Gromov-Hausdorff
      approximations $G:B_i\to B(0,1)$ and $G':B_i' \to B(0,1)$
      satisfying $G(o)=0, G'(o')=0$ such
      that viewing $B(0,1)\subset\mathbf{C}^{n+1}$ under the standard
      embedding, we have $|G(x) - F_i(x)| < \Psi(i^{-1})$ and $|G'(x')
      - g_i\circ F_{j(i)}'(x')| < \Psi(i^{-1})$. Note that
      $F_i(x)=g_i\circ F_{j(i)}'(x')$ by  assumption, and so $|G(x) -
      G'(x')| < \Psi(i^{-1})$. At the same time, under the cone metric on
      $B(0,1)$, the distance from $0$ is H\"older continuous with
      respect to the   Euclidean distance (it is given up to a factor
      by $|x|^{\frac{n-2}{n-1}}$), so this means
      \[ |d_{B(0,1)}(0, G(x)) - d_{B(0,1)}(0,G'(x'))| <
        \Psi(i^{-1}). \]
      Since $G, G'$ are Gromov-Hausdorff approximations,
      we get
      \[ |d_{\omega_i}(o, x) - d_{\eta_i}(o',x')| < \Psi(i^{-1}) \]
      as claimed.

      The balls $B_i, B_i'$ are both $\Psi(i^{-1})$-
      Gromov-Hausdorff close to the unit ball in $C(Y)$.  For
      sufficiently large $i$ we can then use \cite[Proposition
      3.1]{LSz18} to find K\"ahler potentials $\phi_i, \phi_i'$ for
      $\omega_i, \eta_i$, such that
      \[ |\phi_i - d_{\omega_i}(o,\cdot)^2/2| &< \Psi(i^{-1}), \\
        |\phi_i' - d_{\eta_i}(o',\cdot)^2/2| &< \Psi(i^{-1}). \]
      Using \eqref{eq:d2} 
      this means that on $B_{\omega_i}(o, 1)$ we can write
      \[ \Phi_i^*(\eta_i) = \omega_i + \ddb u_i, \]
      with $\sup_{B_{\omega_i}(o,1)} |u_i| < \Psi(i^{-1})$.
      
      We will next apply Proposition~\ref{prop:decayest}. We can
      assume that the $\lambda$ in the proposition is of the form
      $\lambda = 2^{-m}$ for an integer $m$.
      We can also choose $i_0 > 0$, such that the assumptions of
      Proposition~\ref{prop:decayest} hold for $\omega_i$ and $u_i$ on
      $B_{\omega_i}(o,1)$, for all $i \geq i_0$.
      Let us fix a large $k > 0$, and apply the proposition for
      $i=i_0+km$. We have
      \[ \sup_{B_{\omega_{i_0+km}}(o,1)} |u_{i_0+km}| < \epsilon_k, \]
      where $\lim_{k\to\infty} \epsilon_k=0$. 
      We find a $\beta_k,g_k$ and $u_k'$ such that
      \[\label{eq:bkPhi1} \beta_kg_k^*\Phi_k^*\eta_{i_0+km} =
        \omega_{i_0+km} +  \ddb u_k'\]
      on $B_{\omega_{i_0+km}}(o,1)$ together with the estimate
      \[ \sup_{B_{\omega_{i_0+km}}(o,\lambda)} |u_k'| \leq
        \lambda^{2+\alpha}\epsilon_k. \]
      Note that $B_{\omega_{i_0+km}}(o,\lambda) =
      B_{\omega_{i_0+(k-1)m}}(o,1)$. Scaling \eqref{eq:bkPhi1} up by a
      factor of $\lambda^{-2}$, we have
      \[ \lambda^{-2}\beta_k^*g_k^*\Phi_k^*\eta_{i_0+km} = \omega_{i_0+(k-1)m}
        + \ddb \lambda^{-2}u_k', \]
      and $\sup_{B_{\omega_{i_0+(k-1)m}}(o,1)} |\lambda^{-2}u_k'| \leq
      \lambda^\alpha\epsilon_k\leq \epsilon_k$, where we dropped
      the $\lambda^\alpha$ factor.  Note that by construction the
      volume forms of
      $\lambda^{-2}\beta_k^*g_k^*\Phi_k^*\eta_{i_0+km}$
      and $\omega_{i_0+(k-1)m}$ are equal, and so we can apply
      Proposition~\ref{prop:decayest} again, iterating the above
      argument. After $k$ steps we obtain a constant $\Lambda_k$, a
      biholomorphism $G_k: X\to X'$ satisfying $G_k(o)=o'$,
      and a function $U_k$ such that
      \[ G_k^*(\Lambda_k \eta) = \omega_{i_0} + \ddb U_k \]
      on $B_{\omega_{i_0}}(o,1)$, together with the estimate
      \[ \sup_{B_{\omega_{i_0}}(o,1)} |U_k| \leq \epsilon_k, \]
      such that
      \[ (\omega_{i_0} + \ddb U_k)^n = \omega_{i_0}^n. \]
      Here we have absorbed the additional scaling between $\eta$
      and $\eta_k$ into the constant $\Lambda_k$. Note that fixing $i_0$
      we can take $k\to\infty$, and once $\epsilon_k$ is sufficiently
      small, we can apply Savin's small perturbation
      result~\cite{Sav07} to find that on $B_{\omega_{i_0}}(o,1/2)$ we
      have $U_k\to 0$ in $C^\infty$. Since $G_k(o)=o'$,
      we find that for sufficiently large $k$
      \[ B_{\Lambda_k\eta}(o',
        1/4) \subset G_k(B_{\omega_{i_0}}(o,1/2)) \subset B_{\Lambda_k\eta}(o',
        1), \]
      and moreover $G_k^*(\Lambda_k\eta)\to \omega_{i_0}$ in
      $C^\infty$ on
      $B_{\omega_{i_0}}(o,1/2)$. If $\Lambda_k\to 0$, then this is a
      contradiction, since $\eta$ is not flat, and so the
      curvature of $B_{\Lambda_k\eta}(o',1/4)$ blows up as
      $\Lambda_k\to 0$. Similarly
      $\Lambda_k\to\infty$ leads to a contradiction since $\omega_0$ is
      not flat. Choosing a subsequence we can assume $\Lambda_k \to
      \Lambda_\infty > 0$. It follows that we can then take a limit
      $G_k\to G_\infty$ on $B_{\omega_{i_0}(o,1/4)}$ which gives a
      holomorphic   isometry
      \[ G_\infty : B_{\omega_{i_0}}(o,1/4) \to
        B_{\Lambda_\infty\eta}(o',1/4). \]
      We can repeat the same argument for any $i > i_0$, and
      extract a global holomorphic isometry between $(X,\omega_0)$ and
      $(X', \Lambda\eta)$ for a suitable $\Lambda$.

      \section{Further directions}\label{sec:further}
      The approach that we used to prove Theorem~\ref{thm:main} can be
      applied in more general situations. One natural generalization
      would be to study the uniqueness of all of the metrics
      constructed in the author's work \cite{Sz17}, or by
      Conlon-Rochon~\cite{CR17}, given their tangent cones. The places
      where we used the specific choice $\mathbf{C}\times A_1$ for the
      tangent cone were in Lemma~\ref{lem:vbeta} in order to understand
      the quadratic growth harmonic functions, and in
      Proposition~\ref{prop:goodembed} which allowed us to construct
      embeddings of a given Calabi-Yau space as a specific
      hypersurface. When we consider more general tangent cones, then
      these results need to be suitably modified. We expect that in
      general the Calabi-Yau metric with a given tangent cone is not
      unique, however we hope that our methods can be used to
      describe the moduli space of such metrics. 

      To illustrate this, let us
      consider the next simplest example, namely the metric
      $\omega_0$ on   $\mathbf{C}^3$ with tangent cone $\mathbf{C}\times
      A_2$ at infinity, constructed by viewing
      $\mathbf{C}^3\subset\mathbf{C}^4$ as the hypersurface
      \[ \label{eq:A2} z + x_1^2 + x_2^2 + y^3 = 0. \]
      This metric has the property that we have holomorphic functions
      $z,x_1,x_2,y$ whose degrees satisfy $d(z)=1,d(x_1)=3,
      d(x_2)=3, d(y)=2$, and which satisfy \eqref{eq:A2}. Suppose
      now that $(X,\eta)$ is another Calabi-Yau metric with the same
      tangent cone,  and we try to argue as in
      Proposition~\ref{prop:goodembed}. The same arguments show
      that we can embed $X$ into $\mathbf{C}^4$ as a hypersurface
      given by a linear equation in monomials of total degree at most
      6. Moreover this
      equation is a small perturbation of the equation $x_1^2+x_2^2 +
      y^3=0$ defining the tangent cone. As in
      Proposition~\ref{prop:goodembed} we can perform simplifications,
      but we cannot always reduce to the equation \eqref{eq:A2}. Instead
      we may end up with an equation of the form
      \[ az + by + x_1^2 + x_2^2 + y^3 = 0, \]
      for small constants $a,b$, with $a\ne 0$. Note that this defines
      a hypersurface biholomorphic to $\mathbf{C}^3$. When $b\ne 0$, then
      we cannot make the change of coordinate $z' = z + a^{-1}by$ to
      reduce to an equation of the form \eqref{eq:A2}, since $y$ has
      faster growth than $z$. Indeed, we expect that one can construct a
      one-parameter family of inequivalent
      Calabi-Yau metrics on $\mathbf{C}^3$ with tangent cone
      $\mathbf{C}\times A_2$, using the methods from
      \cite{Li17, CR17, Sz17}, viewing
      $\mathbf{C}^3\subset\mathbf{C}^4$ as the hypersurface
      \[ z + by + x_1^2 + x_2^2+y^3=0. \]
      More precisely we expect the following.
      \begin{conj}\label{conj:1} Up to scaling and isometry there is a one parameter
        family of Calabi-Yau metrics on $\mathbf{C}^3$ with tangent
        cone $\mathbf{C}\times A_2$ at infinity.
      \end{conj}
      In view of the
      gluing construction by Li~\cite{Li18} of collapsing Calabi-Yau
      metrics on threefolds (see the discussion in Section 4.2), such
      metrics could  arise as a suitable blowup limit of a collapsing family
      of CY metrics on a threefold that has a fibration locally of the
      form $(x_1,x_2,y) \mapsto x_1^2 + x_2^2 + y^3$.

      A different generalization would be to consider Calabi-Yau
      metrics on more general spaces than $\mathbf{C}^n$. The crucial
      prerequisite for applying the Donaldson-Sun theory in
      Section~\ref{sec:embed}, as well as \cite[Proposition
      3.1]{LSz18}, was that the
      metric $\omega_0$ is $\partial\bar\partial$-exact, and we
      expect that our methods can be extended to classifying such
      exact Calabi-Yau metrics. For instance the smoothing
      $Q^n\subset\mathbf{C}^n$  of the $n$-dimensional $A_1$ singularity,
      \[ 1 + x_1^2 + \ldots x_{n+1}^2 = 0, \]
      is expected to admit a Calabi-Yau metric with tangent cone
      $\mathbf{C}\times A_1$ in terms of the $(n-1)$-dimensional $A_1$
      singularity (see e.g. \cite[Section 4.2]{Li18}), and the
      methods used in Theorem~\ref{thm:main} could lead to a
      uniqueness result for this metric.

      More generally, from the argument in
      Proposition~\ref{prop:goodembed} we can read off which manifolds
      can admit a $\partial\bar\partial$-exact
      Calabi-Yau metric with a given tangent cone. For instance the
      following is a natural conjecture to make.
      \begin{conj}\label{conj:CA1}
          Let $n>4$. The only $\partial\bar\partial$-exact
          Calabi-Yau manifolds of dimension $n$ with tangent cone
          $\mathbf{C}\times A_1$ are $\mathbf{C}\times Q^{n-1}$,
          $\mathbf{C}^n$ and $Q^n$. Moreover up to scaling and
          isometry each of these manifolds admits a unique such
          Calabi-Yau metric. 
        \end{conj}
        The three cases correspond to the function
        $\tilde{f}_i$ in the equation analogous to
        \eqref{eq:reducedeq} having degree 0, 1 or 2. When
        $n\leq 4$ then there would be more possibilities. For both
        Conjectures~\ref{conj:1} and \ref{conj:CA1} we expect that the
        proof of Theorem~\ref{thm:main}
        can be extended to prove the classification results,
        once the corresponding existence results are shown using the
        techniques of \cite{Li17, CR17, Sz17}. 

      Note that some of the results of Donaldson-Sun~\cite{DS15} can
      also be extended to the case when the metric is not exact, under
      the assumption that the tangent cone is smooth away from the
      vertex (see Liu~\cite{Liu17}). This is closer to the
      setting of asymptotically conical Calabi-Yau metrics considered
      by Conlon-Hein~\cite{CH14} who also obtained classification
      results for Calabi-Yau metrics with prescribed tangent cone. At
      the moment there is little that we can say in this direction
      about general Calabi-Yau manifolds with tangent cones that have non-isolated
      singularities, beyond the result in \cite{LiuSz2} that each
      tangent cone is a normal affine variety. 

\bibliography{../../mybib}
\bibliographystyle{amsplain}

\end{document}